\documentclass[journal,romanappendices]{IEEEtran}

\usepackage{amsmath, amsfonts, amssymb, mathtools}
\usepackage{algorithm}
\usepackage[noend]{algpseudocode}
\usepackage{graphicx, epsfig}
\usepackage{subcaption}
\usepackage{cite}
\usepackage[inline]{enumitem}

\usepackage{tikz}
\usetikzlibrary{positioning, arrows.meta}
\usetikzlibrary{arrows,calc,fit}

\tikzset{block/.style={
            draw=black,
            thick,
            align=center,
            rounded corners=7pt,
            minimum width=7cm,
            minimum height=0.9cm
        }}
\tikzset{line/.style={
            draw, 
            thick, 
            -Stealth
        }}

\newtheorem{theorem}{Theorem}

\newtheorem{lemma}{Lemma}

\newtheorem{assumption}{Assumption}

\newtheorem{remark}{Remark}

\DeclareMathOperator*{\argmin}{arg\,min}

\DeclareMathOperator*{\st}{subject\;to\;}
\DeclareMathOperator*{\sst}{s.t.\;}

\DeclareMathOperator{\gra}{gra}

\DeclarePairedDelimiter{\product} {\langle} {\rangle}
\DeclarePairedDelimiter{\norm} {\lVert} {\rVert}
\DeclarePairedDelimiter{\positive} {[} {]_+}

\renewcommand{\IEEEQED}{\IEEEQEDopen}
\newcommand{\proj}{\mathcal{P}}

\title{Distributed Optimization with Coupled Constraints over Time-Varying Digraph}
\author{Yeong-Ung Kim, and Hyo-Sung Ahn
    \thanks{This work was supported by the National Research Foundation of Korea (NRF) grant funded by the Korea government (MSIT) (2022R1A2B5B03001459)
    }
    \thanks{The authors are
        with the Department of Mechanical and Robotics Engineering, 
        Gwangju Institute of Science and Technology(GIST), Gwangju, 61005, Republic of Korea
        (e-mails: {\tt\small yeongungkim@gm.gist.ac.kr, hyosung@gist.ac.kr}).
    }
}

\begin{document}
\maketitle

\begin{abstract}
    In this paper, we develop a distributed algorithm for solving a class of distributed convex optimization problems where the local objective functions can be a general nonsmooth function, and all equalities and inequalities are network-wide coupled.
    This type of problem arises from many areas, such as economic dispatch, network utility maximization, and demand response.
    Integrating the decomposition by right hand side allocation and primal-dual methods, the proposed algorithm is able to handle the distributed optimization over networks with time-varying directed graph in fully distributed fashion.
    This algorithm does not require the communication of sensitive information, such as primal variables, for privacy issues. 
    Further, we show that the proposed algorithm is guaranteed to achieve an $O(1/k)$ rate of convergence in terms of optimality based on duality analysis under the condition that local objective functions are strongly convex but not necessarily differentiable, and the subdifferential of local inequalities is bounded.
    We simulate the proposed algorithm to demonstrate its remarkable performance.
\end{abstract}

\section{Introduction}
Multiagent systems are being actively researched in fields such as autonomous vehicle swarms and collaborative robots.
Among the studies on multiagent systems, one of the most important researches is distributed optimization, in which a network of agents collaboratively seeks an optimal value to minimize a global objective function while relying only on local computation and communication with their interconnected neighbors~\cite{zheng2022review, yang2019survey}.
Distributed optimization is used for various applications such as economic dispatch in power systems~\cite{mao2020privacy, li2018accelerated}, multi-robot task allocation~\cite{zavlanos2008distributed}, and optimization based multiagent control~\cite{chen2020online, tan2021distributed}.
The study of distributed optimization mainly investigates two types of problems, consensus optimization~\cite{nedic2009distributed} and constraint coupled optimization~\cite{doostmohammadian2025survey}.

Consensus optimization is formulated as minimizing the sum of local objective functions subject to a global constraint that all agents must agree on a common decision variable.
The consensus constraint is typically addressed by interleaving local optimization steps with a distributed averaging and consensus process.
The communication of network is modeled as either an undirected or a directed graph whose edges represent available connections for transferring information between agents.
The algorithms with undirected graph, such as DIGing~\cite{nedic2017achieving}, EXTRA~\cite{shi2015extra} and P-EXTRA~\cite{shi2015proximal}, are studied.
Additionally, algorithms based on the concept of duality are also explored in \cite{xu2018bregman, maros2020geometrically, ghaderyan2023fast}.
Subsequent studies have focused on handling directed network topologies, where the methods \cite{tsianos2012push, nedic2014distributed} based on push-sum protocol~\cite{kempe2003gossip} are studied.
The authors in \cite{nedic2017achieving, pu2020push} combine gradient tracking with the push-sum protocol.
Some studies also consider a constraint which is either identical constraint known to all agents \cite{li2018distributed, xi2016distributed} or local constraints \cite{nedic2010constrained, margellos2017distributed, wu2019fenchel, shahriari2025double}.

Constraint coupled optimization is another type of distributed optimization problem.
The core challenge lies in the distributed nature of the network: while the objective function is separable, the coupling constraint links the decision variables of all agents. 
As it is more challenging and increasingly relevant problem in modern system, such as resource allocation problem and cooperative control system, it has garnered significant attention over the past decade \cite{lakshmanan2008decentralized, zhang2017distributed, falsone2017dual, nedic2018decentralized, xu2018dual, aybat2019distributed, notarnicola2019constraint, camisa2021distributed, wu2022distributed}.
To overcome the difficulty imposed by the non-separable coupling constraint, distributed methodologies rely on optimization duality and decomposition techniques.

By incorporating the coupling constraint into the Lagrangian, the complex problem can be mathematically decomposed into a sequence of simpler, uncoupled local subproblems that each agent can solve independently. 
In \cite{falsone2017dual}, distributed algorithm has been proposed based on a dual proximal optimization algorithm employing running averages for primal recovery.
The author in \cite{notarnicola2019constraint} introduces auxiliary neighboring variables to mimic a local version of the coupling constraint and a relaxation to allow positive violation for the subproblem minimizing the local objective subject to local uncoupled constraint.
In \cite{camisa2021distributed}, the primal decomposition method, called right-hand side allocation, is used in order to extend the analysis in \cite{notarnicola2019constraint} to time-varying undirected graph and reduce the number of communication steps.
This approach involves each agent holding a local allocation of the total resource, and through iterative processes of bargaining allocations with neighbors, the optimal allocation is gradually discovered.
Since the Lagrange dual problem of constraint coupled optimization is closely related to consensus optimization, the algorithms for consensus algorithm can be applied to the dual problem of constraint coupled optimization.
The algorithm proposed by \cite{xu2018dual} is an extension of the similar idea of splitting methods that appeared in \cite{xu2018bregman} to the dual problem.
The author in \cite{wu2022distributed} proposes a distributed algorithm over undirected graph based on P-EXTRA \cite{shi2015proximal} which has $O(1/k)$ rate.

In this article, we address a general constraint coupled optimization problem over time-varying directed graph.
We attempt to decompose the problem based on right-hand side allocation and solve the problem integrating augmented Lagrangian method and dual approach.
As a result, we propose a novel distributed algorithm utilizing only one doubly stochastic matrix and a fixed step size.
Note that the proposed algorithm does not require exchanges of the privacy information, such as primal variable, gradient, and its contribution to coupled constraints.
Consequently, the main contributions of this paper could be summarized as follows:
\begin{itemize}
    \item First, we propose a novel distributed algorithm that solves constraint coupled optimization problem.
        The proposed algorithm considers a time-varying and directed communication graph which is highly challenging.
        In each iteration, the algorithm requires only one communication with a doubly stochastic matrix and a fixed step size.
        The proposed algorithm exchanges only dual information and can be utilized even in privacy-sensitive environments.
    \item Second, we rigorously prove that, under the assumptions of strong convexity of the objective functions and bounded subgradients of the inequality constraints, the algorithm converges to the optimal solution with a convergence rate of $O(1/k)$ in the dual aspect.
        This rate establishes the efficiency of our approach in achieving optimal solutions rapidly.
\end{itemize}

This paper is organized as follows:
Section~\ref{sec:problem} introduces the convex optimization problem subject to globally coupled constraints and present our distributed optimization algorithm.
Section~\ref{sec:convergence} is dedicated to the convergence analysis of the proposed algorithm, and the numerical example is provided in Section~\ref{sec:simulation}.
Finally, Section~\ref{sec:conclusion} concludes this article.

\emph{Notation}:
Let $\mathbb{R}$, $\mathbb{R}_{+}$ and $\mathbb{Z}_{+}$ denote the sets of real, nonnegative real, nonnegative integer numbers, respectively.
The $n \times n$ identity matrix, and the vector of zero and one in $\mathbb{R}^n$ are denoted by $I_n$, $0_n$ and $1_n$, respectively, where the dimension subscripts are dropped when context makes them clear.
The notations $\lVert \cdot \rVert$ and $\lVert \cdot \rVert_1$ denote the Euclidean norm and the $\ell_1$ norm, respectively.
For any matrix $A \in \mathbb{R}^{n \times m}$, $A_{ij}$, $A_{i:}$ and $A_{:j}$ denote $(i,j)$-entry, $i$th row vector, and $j$th column vector of $A$, respectively.
The transpose of a vector $x$ and a matrix $A$ are denoted by $x^T$ and $A^T$, respectively.
The Kronecker product of two matrices $A$ and $B$ is denoted by $A \otimes B$.
Given a couple of vectors $x_1, \dots, x_n$, the notation $(x_1, \dots, x_n)$ presents the concatenated vector obtained by stacking, i.e., $(x_1, \dots, x_n) = {[x_1^T, x_2^T, \dots, x_n^T]}^T$.
For two vectors $x$ and $y$, $x \le y$ denotes an element-wise comparison.
The inner product is denoted by $\langle \cdot, \cdot \rangle$.
For any function $f : \mathbb{R}^d \to \mathbb{R}$, $\partial f(x)$ denotes the subdifferential of $f$ at $x$, and $\nabla f(x)$ denotes the gradient of $f$ at $x$ if $f$ is differentiable.
The projection of $x$ on $M$ is denoted $\proj_M(x) = \argmin_{z \in M} \lVert z - x \rVert$.

\section{Problem, Algorithm, and Assumptions}\label{sec:problem}
We consider a network of $N$ agents and the following constrained convex optimization problem over a time-varying directed graph $\mathcal{G}_k = (\mathcal{V}, \mathcal{E}_k)$ with vertex set $\mathcal{V} = \{1, \dots, N\}$ and edge set $\mathcal{E}_k \subset \mathcal{V} \times \mathcal{V}$ at time $k \in \mathbb{Z}_{+}$:
\begin{equation}\label{eq:original_problem}
    \begin{aligned}
        &\min_{x_i \in \mathbb{R}^{d_i}, \forall i \in \mathcal{V}} && f(\mathbf{x}) = \sum_{i \in \mathcal{V}} f_i(x_i) \\
        &\st && \sum_{i \in \mathcal{V}} A_i x_i = \sum_{i \in \mathcal{V}} b_i, \quad \sum_{i \in \mathcal{V}} g_i(x_i) \le 0,
    \end{aligned}
\end{equation} 
where $x_i \in \mathbb{R}^{d_i}$ is the local decision variable of agent $i$ and $\mathbf{x} = (x_1,\dots,x_N) \in \mathbb{R}^{\sum d_i}$.
The variables $A_i \in \mathbb{R}^{p \times d_i}$, $b_i \in \mathbb{R}^p$ and the function $g_i = (g_{i1}, \dots, g_{iq}): \mathbb{R}^{d_i} \rightarrow \mathbb{R}^{q}$ are $i$th contribution of the globally coupled constraints.
The functions $f_i$, $g_i$, the matrix $A_i$ and vector $b_i$ are privately known by agent $i$ only.
To guarantee the problem is solvable and has strong duality, the following assumption is made throughout the paper.
\begin{assumption}[Convexity] \label{asm:convexity}
    \begin{enumerate*}[label=(\alph*)]
        \item All $f_i: \mathbb{R}^{d_i} \rightarrow \mathbb{R}$ are $\mu$-strongly convex.
        \item Each $g_i$ is a component-wise convex function.
    \end{enumerate*}
\end{assumption}
\begin{assumption}[Subgradient Boundedness] \label{asm:bound_subgradient_ineq}
    There exists $L_g > 0$ such that $\norm{s} \le L_g$ for all $s \in \partial g_{il}(x)$, $x \in \mathbb{R}^{d_i}$, $i \in \mathcal{V}$, and $l = 1,\dots,q$. 
\end{assumption}
\begin{assumption}[Slater's condition]\label{asm:slater}
    There exists a feasible point $\tilde{\mathbf{x}}= (\tilde{x}_1, \dots, \tilde{x}_N)$ such that $\sum_{i \in \mathcal{V}} A_i \tilde{x}_i - b_i = 0$ and $\sum_{i \in \mathcal{V}} g_i(\tilde{x}_i) < 0$.
\end{assumption}
Moreover, we consider the following network assumption related to the communication on the time-varying directed graph.
\begin{assumption}
    For all $k \in \mathbb{Z}_+$,
    \begin{enumerate*}[label=(\alph*)]
        \item \emph{(Strong connectivity)} $\mathcal{G}_k$ is strongly connected, and
        \item \emph{(Aperiodicity)} $\mathcal{G}_k$ has self-loops, i.e., $(i,i) \in \mathcal{E}_k$ for all $i \in \mathcal{V}$.
    \end{enumerate*}
\end{assumption}
The in-neighbor set of agent $i$ in graph $\mathcal{G}_k$ at time $k$ is defined as $\mathcal{N}_{\text{in}}^k(i) = \{ j \in \mathcal{V} \mid (j, i) \in \mathcal{E}_k \}$.

\subsection{Problem Reformulation}
We reformulate \eqref{eq:original_problem} to develop a distributed algorithm.
Introduced in \cite{bertsekas2016nonlinear}, utilizing auxiliary variables $v_i, z_i$ for $i \in \mathcal{V}$, the problem \eqref{eq:original_problem} can be equivalently converted to
\begin{equation}\label{eq:reformed_problem}
    \begin{aligned}
        &\min_{\substack{\{x_i \in \mathbb{R}^{d_i}\}_{i \in \mathcal{V}} \\ \{v_i\}_{i\in\mathcal{V}}, \{z_i\}_{i\in\mathcal{V}}}} &&\sum_{i \in \mathcal{V}} f_i(x_i) \\
        &\text{subject to}\; && A_i x_i - b_i = v_i, \;\; g_i(x_i) \le z_i, \\
        &&& \sum_{i\in\mathcal{V}} v_i = 0,\;\; \sum_{i\in\mathcal{V}} z_i = 0.
    \end{aligned}
\end{equation}
The auxiliary variables decouple the dense constraints in \eqref{eq:original_problem} into independent constraints and zero-sum constraints.
Let $x_i^*, v_i^*, z_i^*$, for $i \in \mathcal{V}$ be an optimal solution of the problem.
The corresponding Karush-Kuhn-Tucker (KKT) condition can be divided into the local conditions and the global conditions.
The local conditions for agent $i$ are given by the system:
\begin{equation}\label{eq:kkt_local}
    \begin{aligned} 
        \partial f_i(x_i^*) + A_i^T u_i^* + \sum_{l = 1}^q y_{il}^* \partial g_{il}(x_i^*) &\ni 0, \\
        A_i x_i^* - b_i - v_i^* &= 0, \\
        \langle y_i^*, g_i(x_i^*) - z_i^* \rangle &= 0, \\
        y_i^* &\ge 0,
    \end{aligned}
\end{equation}
where $u_i^*$ and $y_i^*$ are local Lagrange multipliers associated with agent $i$'s independent equality and inequality constraints, respectively.
The global conditions are the network-wise coupled zero-sum constraints and corresponding constraints for the multipliers:
\begin{equation}\label{eq:kkt_global}
    \begin{aligned}
        u^* := u_1^* = u_2^* = \cdots &= u_N^*, \\
        y^* := y_1^* = y_2^* = \cdots &= y_N^*, \\
        v_1^* + \dots + v_N^* &= 0, \\
        z_1^* + \dots + z_N^* &= 0.
    \end{aligned}
\end{equation}

\subsection{Algorithm Development}
The distributed algorithm for solving~\eqref{eq:reformed_problem} has to ensure two conditions: 
\begin{enumerate*}[label=\arabic*)]
    \item the multipliers should be in the consensus space $\mathcal{C}_N$ where all elements have the same value, i.e., $(r_1, \dots, r_N) \in \mathcal{C}_N$ implies $r_1 = \cdots = r_N$, and
    \item the auxiliary variables should be in the zero-sum space $\mathcal{C}_N^\perp$, as described in KKT conditions~\eqref{eq:kkt_global}.
\end{enumerate*}
We proceed by decomposing optimization problem~\eqref{eq:reformed_problem} into a two-level structure featuring independent local subproblems and a master problem:
\begin{equation}\label{eq:master}
    \begin{aligned}
        &\min_{\{v_i\}_{i\in\mathcal{V}}, \{z_i\}_{i\in\mathcal{V}}} && \sum_{i\in\mathcal{V}} h_i(v_i, z_i) \\
        &\text{subject to} && \sum_{i\in\mathcal{V}} v_i = 0,\;\; \sum_{i\in\mathcal{V}} z_i = 0, \\
    \end{aligned}
\end{equation}
where the $i$th subproblem which is defined the primal function, is given by
\begin{equation}\label{eq:subproblem}
    \begin{aligned}
        h_i(v_i, z_i) =\; &\min_{x_i} && f_i(x_i) \\
                        &\sst && A_i x_i - b_i = v_i,\;\; g_i(x_i) \le z_i.
    \end{aligned}
\end{equation}
If $h_i$ is finite at $(v_i, z_i)$, then the subdifferential of $h_i$ at $(v_i, z_i)$ is equal to the set of geometric multipliers for minimizing $f_i(x)$ subject to $A_i x - b_i = v_i$ and $g_i(x) \le z_i$.
This subdifferential of $h_i$ at $(v_i, z_i)$ is nonempty if the feasible set $\{x \in \mathbb{R}^{d_i} \mid A_i x - b_i = v_i, g_i(x) \le z_i  \}$ is nonempty \cite{bertsekas2003convex}.

As in \cite{bertsekas2016nonlinear}, utilizing the projected subgradient method for the master problem yields the following update with a parameter $\gamma$:
\begin{equation}\label{eq:global_subgradient}
    \begin{aligned}
        \mathbf{v}^{k+1} &= \proj_{\mathcal{C}_N^\perp}(\mathbf{v}^k + \gamma \mathbf{u}^{k*}), \\
        \mathbf{z}^{k+1} &= \proj_{\mathcal{C}_N^\perp}(\mathbf{z}^k + \gamma \mathbf{y}^{k*}),
    \end{aligned}
\end{equation}
where the bold symbol represents the stacking of the corresponding vector, i.e., $\mathbf{v}^k = (v_1^k, \dots, v_N^k)$, and similarly for other variables.
However, we can not guarantee geometric multipliers $u_i^{k*}$ and $y_i^{k*}$ exist for all $v_i^k$ and $z_i^k$ \cite{bertsekas2016nonlinear}, \cite[Remark 2.9]{camisa2021distributed} and converge to consensus space $\mathcal{C}_N$.
In order to handle these issues, we exploit the augmented Lagrangian method for approximating geometric multipliers with a consensus process.

We first define a time-varying matrix $W^k \in \mathbb{R}_{+}^{N \times N}$ corresponding to the time-varying graph $\mathcal{G}_k$ satisfying the following assumption:
\begin{assumption}[Weight Matrices]\label{asm:weight_mat}
    \begin{enumerate*}[label=(\alph*)]
        \item The matrix $W^k$ has the same structure corresponding to the time-varying digraph $\mathcal{G}_k = (\mathcal{V}, \mathcal{E}_k)$ for $k \ge 1$, i.e., $W^k_{ij} > 0$, for all $(j,i) \in \mathcal{E}_k$ and $W^k_{ij} = 0$, otherwise.
        \item $W^k$ is doubly stochastic matrix, i.e., $W^k 1_N = {W^k}^T 1_N = 1_N$ and $W^0 = I$.
    \end{enumerate*}
\end{assumption}
At iteration $k$ for agent $i$, the approximation of the subgradient of $h_i$ is computed by
\begin{equation*}
    \begin{aligned}
        p_i^k &= \sum_{j \in \mathcal{N}_{\text{in}}^k(i)} W^k_{ij} u_j^k,\\
        q_i^k &= \sum_{j \in \mathcal{N}_{\text{in}}^k(i)} W^k_{ij} y_j^k,\\
        x_i^{k+1} &= \argmin_{x_i} \mathcal{L}^i_\rho (x_i, p_i^k, q_i^k, v_i^k, z_i^k),\\
        u_i^{k+1} &= p_i^k + \rho (A_i x_i^{k+1} - b_i - v_i^k), \\
        y_i^{k+1} &= \left[ q_i^k + \rho (g_i(x_i^{k+1}) - z_i^k) \right]_+,
    \end{aligned}
\end{equation*}
where the augmented Lagrangian function $\mathcal{L}_\rho^i$ is given by
\begin{equation*}
    \begin{aligned}
        &\mathcal{L}_\rho^i (x_i, p_i^k, q_i^k, v_i^k, z_i^k) \\
        &= f_i(x_i) + \langle p_i^k, A_i x_i - b_i - v_i^k \rangle + \frac{\rho}{2} \lVert A_i x_i - b_i - v_i^k \rVert^2 \\
        &\quad + \frac{1}{2\rho} \left\lVert [ q_i^k + \rho (g_i(x_i) - z_i^k) ]_+ \right\rVert^2 - \frac{1}{2\rho} \lVert q_i^k \rVert^2,
    \end{aligned}
\end{equation*}
and $[\cdot]_+$ denotes the projection on the first quadrant.

The approximated subgradient now is obtained, but since \eqref{eq:global_subgradient} relies on the global information, the projection on the zero-sum space needs to be replaced by a suitable distributed operation.
Since, from the doubly stochasticity of $W$, $1^T (I - W) = 0$.
Leveraging this property, with the additional assumption of initial value $v_i^0 = 0$ and $z_i^0 = 0$, we approximate the projection as the procedure \eqref{eq:approx_proj_start}-\eqref{eq:approx_proj_end}.
It is equal to
\begin{equation*}
    \begin{aligned}
        \mathbf{v}^{k+1} &= \mathbf{v}^k + \gamma ((I - W^{k+1}) \otimes I_p) \mathbf{u}^{k+1}, \\
        \mathbf{z}^{k+1} &= \mathbf{z}^k + \gamma ((I - W^{k+1}) \otimes I_q) \mathbf{y}^{k+1},
    \end{aligned}
\end{equation*}
then, for all $k \in \mathbb{Z}_+$, the auxiliary variables always remain in the zero-sum space $\mathcal{C}_N^\perp$, i.e., $(1_N \otimes \eta_p)^T \mathbf{v}^k = 0$ and $(1_N \otimes \eta_q)^T \mathbf{z}^k = 0$ with arbitrary $\eta_p \in \mathbb{R}^p$ and $\eta_q \in \mathbb{R}^q$.
The overall update rule is summarized in Algorithm~\ref{alg:DCO}.

\begin{algorithm}[t]
    \caption{Distributed Coupled Optimization}
    \label{alg:DCO}
    \begin{algorithmic}
        \State \textbf{Initialization:}
        \State \hspace{1em} Arbitrary $u_i^0 = p_i^0,\; y_i^0 = q_i^0$ and 
        \State \hspace{1.5em} $v_i^0 = 0,\; z_i^0 = 0, \; i = 1,\dots,N$.
        \Loop
            \ForAll{$i \in \{1, \dots, N\}$}
                \State \makebox[0.875\linewidth]{$x_i^{k+1} \gets \argmin_{x_i} \mathcal{L}_\rho^i (x_i, p_i^k, q_i^k, v_i^k, z_i^k)$\hfill\refstepcounter{equation}\llap{(\theequation)}\label{eq:algo_primal}}
                \State \makebox[0.875\linewidth]{$u_i^{k+1} \gets p_i^k + \rho(A_i x_i^{k+1} - b_i - v_i^k)$\hfill\refstepcounter{equation}\llap{(\theequation)}\label{eq:algo_dual_eq}}
                \State \makebox[0.875\linewidth]{$y_i^{k+1} \gets \positive{q_i^k + \rho (g_i(x_i^{k+1}) - z_i^k)}$\hfill\refstepcounter{equation}\llap{(\theequation)}\label{eq:algo_dual_inq}}
                \State
                \State Synchronization
                \State
                \State \makebox[0.875\linewidth]{$p_i^{k+1} \gets \sum_j W^{k+1}_{ij} u_j^{k+1}$\hfill\refstepcounter{equation}\llap{(\theequation)}\label{eq:approx_proj_start}}
                \State \makebox[0.875\linewidth]{$q_i^{k+1} \gets \sum_j W^{k+1}_{ij} y_j^{k+1}$\hfill\refstepcounter{equation}\llap{(\theequation)}}
                \State 
                \State \makebox[0.875\linewidth]{$v_i^{k+1} \gets v_i^k + \gamma (u_i^{k+1} - p_i^{k+1})$\hfill\refstepcounter{equation}\llap{(\theequation)}}
                \State \makebox[0.875\linewidth]{$z_i^{k+1} \gets z_i^k + \gamma (y_i^{k+1} - q_i^{k+1})$\hfill\refstepcounter{equation}\llap{(\theequation)}\label{eq:approx_proj_end}}
            \EndFor
            \State $k \gets k + 1$
        \EndLoop
    \end{algorithmic}
\end{algorithm}

\begin{remark}
    A direct approach to solving \eqref{eq:master} requires adding the constraints $(v_i, z_i) \in Z_i$ for $i \in \mathcal{V}$ where each $Z_i$ is the set of $(v_i, z_i)$ that the subproblem~\eqref{eq:subproblem} is feasible \cite{bertsekas2016nonlinear}.
    In a distributed scheme, it is difficult to simultaneously consider feasible regions $Z_i$ for $i \in \mathcal{V}$ and zero-sum constraints so that we require a relaxation of the primal problem that is amenable to distributed implementation.
    The algorithm proposed in \cite{notarnicola2019constraint,camisa2021distributed} relaxed the constraints allowing a non-negative violation $\rho$ and adding a penalty term $M \rho$ to the objective function, where $M$ is a sufficiently large scalar.
    However, the method using non-negative violations is heavily influenced by the coefficient $M$ and estimates the multiplier more indirectly than the augmented Lagrangian.
\end{remark}

\section{Convergence Analysis}\label{sec:convergence}
In this section, we study the convergence property of Algorithm~\ref{alg:DCO}.
Before proceeding to the main result, let us establish the dual analysis, which are useful in the subsequent analysis.

We define the dual function of the $i$th subproblem with optimal auxiliary variables $v_i^*$ and $z_i^*$:
\begin{equation*}
    \begin{aligned}
        \varphi_i(\zeta) &= \min_{x_i} \mathcal{L}_0^i(x_i, u, y) \\
        &= \min_{x_i} f_i(x_i) + \product{u, A_i x_i - b_i - v_i^*} \\
        &\qquad\qquad + \product{y, g_i(x_i) - z_i^*},
    \end{aligned}
\end{equation*}
where $\zeta = (u, y)$.
By the strong convexity, the updates~\eqref{eq:algo_primal}-\eqref{eq:algo_dual_inq} are equivalent to the proximal maximization of the dual function~\cite{bertsekas2016nonlinear}.
Then, Algorithm~\ref{alg:DCO} can be written as
\begin{subequations}
    \begin{align}
        \zeta_i^{k+1} &= \argmin_{\zeta \in \mathbb{R}^p_{\vphantom{g}} \times \mathbb{R}^q_+} \left\langle \xi_i^k, \zeta \right\rangle - \varphi_i(\zeta) + \frac{1}{2\rho}\lVert \zeta - \hat\zeta_i^k \rVert^2, \label{eq:dual_update} \\
        \hat\zeta_i^{k+1} &= \mathbf{W}^{k+1}_{i:} \boldsymbol{\zeta}^{k+1}, \\
        \xi_i^{k+1} &= \xi_i^k + \gamma (\zeta_i^{k+1}- \hat\zeta_i^{k+1}). \label{eq:dual_end}
    \end{align}
\end{subequations}
where $\hat\zeta_i^k = (p_i^k, q_i^k)$, $\xi_i^k = (v_i^k - v_i^*, z_i^k - z_i^*)$, $\mathbf{W}^{k+1}_{i:} = W^{k+1}_{i:} \otimes I_{p+q}$ and $\boldsymbol{\zeta}^k = (\zeta^k_1,\dots,\zeta^k_N)$.
The first two terms in \eqref{eq:dual_update} is the negative of a dual function for the subproblem $h_i(v_i^k, z_i^k)$, i.e.,
\begin{equation*}
    \begin{aligned}
    \varphi_i(\zeta) - \left\langle \xi_i^k, \zeta \right\rangle = \min_{x_i} &\; f_i(x_i) + \product{u, A_i x_i - b_i - v_i^k} \\
        &\quad + \product{y, g_i(x_i) - z_i^k},
    \end{aligned}
\end{equation*}
and the last proximal term comes from the Fenchel duality of the quadratic penalty term.
We first establish a lemma about the smoothness of the dual functions.

\begin{lemma}[$L$-smoothness of $\varphi_i$]\label{lem:dual_Lipschitz}
    Let Assumption~\ref{asm:convexity},~\ref{asm:bound_subgradient_ineq} hold.
    The dual function $\varphi_i$ has Lipschitz continuous gradient with a constant $L$ for $i = 1,\dots,N$.
\end{lemma}
\begin{IEEEproof}
    By Assumption~\ref{asm:convexity}, the Lagrangian function $\mathcal{L}_0^i(x, u, y)$ is $\mu$-strongly convex.
    We define 
    \begin{equation*}
        \begin{aligned}
            F_i(x, u, y) 
            &= \partial_x \mathcal{L}_0^i(x, u, y) \\
            &= \partial f_i(x) + A_i^T u + \sum_{l = 1}^q y_l \partial g_{il}(x).
        \end{aligned}
    \end{equation*}
    Consider two distinct multipliers $\zeta = (u, y)$ and $\zeta' = (u', y')$, with corresponding minimizers $x$ and $x'$ such that $0 \in F_i(x,u,y)$ and $0 \in F_i(x',u',y')$.
    By the property of strong convexity~\cite[Lemma 3]{zhou2018fenchel}, we have $\mu \lVert x - x' \rVert \le \lVert s - s' \rVert$ for any $s \in F_i(x,u,y)$ and any $s' \in F_i(x',u,y)$.
    Then, since $0 \in F_i(x, u, y)$ and $0 \in F_i(x', u', y')$, we obtain
    \begin{equation*}
        \begin{aligned}
            \lVert x - x' \rVert 
            &\le \frac{1}{\mu} \left\lVert s_f + A_i^T u + \sum_{l = 1}^q y_l s^g_{il} \right\rVert \\
            &\le \frac{1}{\mu} \Bigg\lVert s_f + A_i^T u'+ \sum_{l = 1}^q y'_l s^g_{il} \\
            &\qquad\qquad + A_i^T(u - u') + \sum_{l = 1}^q (y_l  - y'_l) s^g_{il} \Bigg\rVert\\
            &= \frac{1}{\mu} \left\lVert A_i^T(u - u') + \sum_{l = 1}^q (y_l  - y'_l) s^g_{il} \right\rVert \\
            &\le \frac{\sqrt{\norm{A_i}^2 + q L_g^2}}{\mu} \lVert (u - u', y - y') \rVert
        \end{aligned}
    \end{equation*}
    where $s_f \in \partial f_i(x')$ satisfying $s_f + A_i^T u'+ \sum_{l = 1}^q y'_l s^g_{il} = 0 \in F_i(x', u', y')$ and $s^g_{il} \in \partial g_{il}(x')$.
    Since for $0 \in F_i(x,u,y)$
    \begin{equation*}
        \nabla \varphi_i(\zeta) =
        \begin{bmatrix}
            A_i x - b - v_i^* \\
            g_i(x) - z_i^*
        \end{bmatrix},
    \end{equation*}
    we have
    \begin{equation*}
        \begin{aligned}
            \norm{\nabla\varphi_i\left(\zeta \right) - \nabla\varphi_i\left(\zeta'\right)} 
            &\le \lVert A_i (x - x') \rVert + qL_g \lVert x - x' \rVert \\
            &\le L_i \norm{(u, y) - (u', y')}
        \end{aligned}
    \end{equation*}
    where $L_i = (\norm{A_i}^2 + qL_g^2) / \mu$.
    Let $L = \max_i L_i$, then we complete the proof.
\end{IEEEproof}

We now state the main convergence result for the proposed algorithm.
\begin{theorem}\label{thm:convergence}
    Let Assumption~\ref{asm:convexity}-\ref{asm:weight_mat} and the initial condition $(\mathbf{v}^0, \mathbf{z}^0) = 0$ hold.
    Define a geometric multiplier $\zeta^* = (u^*, y^*)$ for \eqref{eq:original_problem}, and $C_0$ be a constant depending on the initial value $u_i^0$ and $y_i^0$ for $i = 1, \dots, N$.
    Then, for sufficiently large $K$, with $\gamma = 1/\rho$, and $\rho < 1/2L$, we have 
    \begin{equation*}
        \sum_{k = 0}^{K-1} \sum_{i=1}^N \left( \varphi_i(\zeta^*) - \varphi_i(\zeta_i^k) \right) \le \frac{C_0}{2\rho}.
    \end{equation*}
\end{theorem}
\begin{IEEEproof}
    See Appendix~\ref{sec:convergence_proof}
\end{IEEEproof}

\begin{remark}
    Let $\bar{\zeta}_i^K = (\bar{u}_i^K, \bar{y}_i^K) = \frac{1}{K} \sum_{k=0}^{K-1} \zeta_i^k$ for $i = 1, \dots, N$.
    Since the dual function $\varphi_i$ is concave, we may use Jensen's inequality and obtain
    \begin{equation*}
        \sum_{i=1}^N \left( \varphi_i(\zeta^*) - \varphi_i(\bar{\zeta}_i^K) \right) \le \frac{C_0}{2\rho K}.
    \end{equation*}
    The statement of Theorem~\ref{thm:convergence} shows that Algorithm~\ref{alg:DCO} makes the dual optimality converge in $O(1/k)$ rate.
\end{remark}

\begin{figure*}[ht]
    \centering
    \begin{subfigure}[b]{0.3\textwidth}
        \centering
        \includegraphics[width=\textwidth]{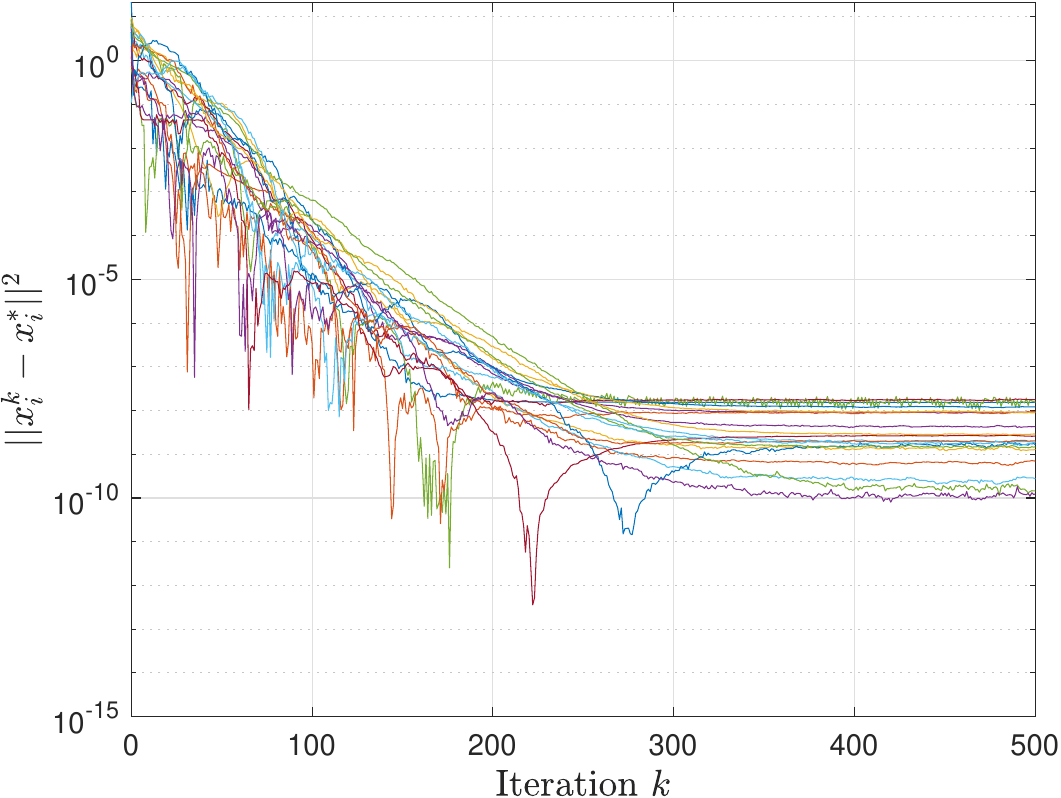}
        \caption{}
        \label{fig:primal}
    \end{subfigure} 
    \hfill
    \begin{subfigure}[b]{0.3\textwidth}
        \centering
        \includegraphics[width=\textwidth]{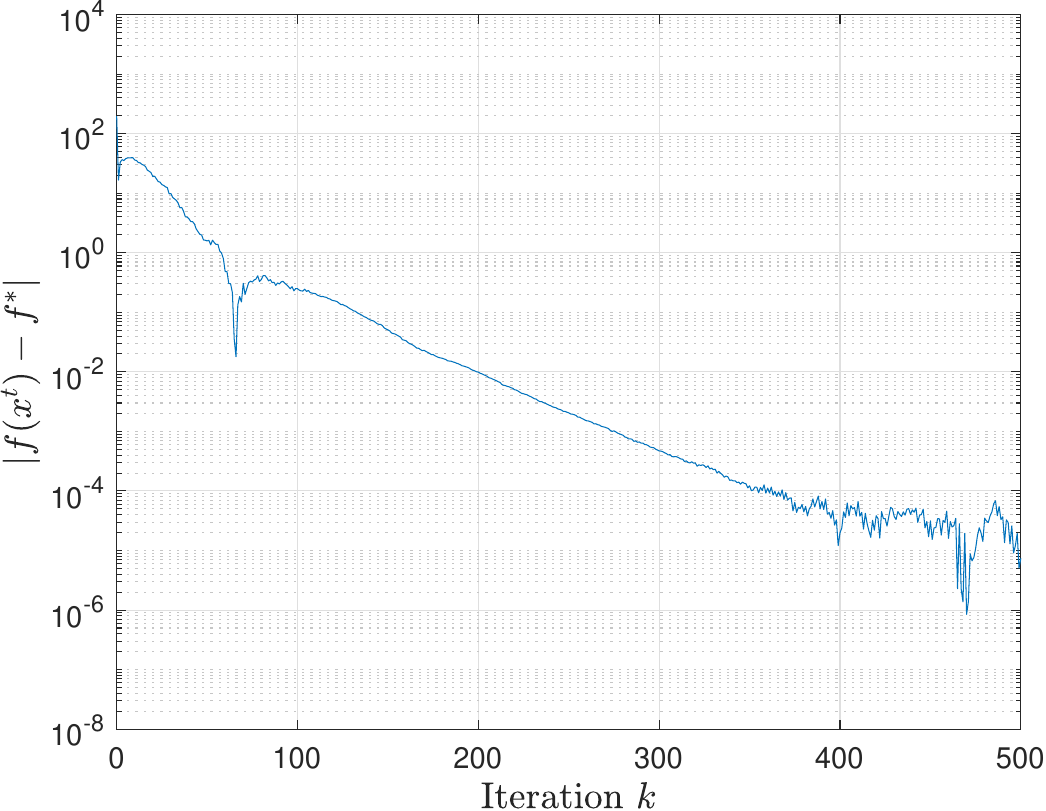}
        \caption{}
        \label{fig:objective}
    \end{subfigure}
    \hfill
    \begin{subfigure}[b]{0.3\textwidth}
        \centering
        \includegraphics[width=\textwidth]{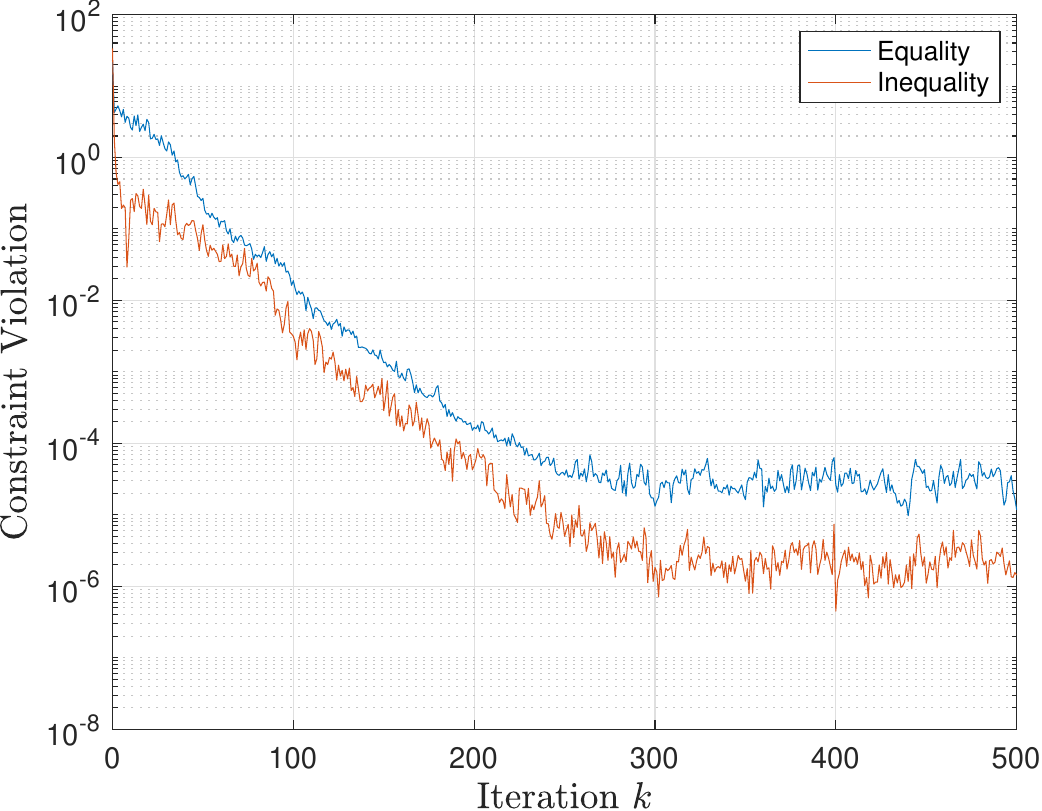}
        \caption{}
        \label{fig:violation}
    \end{subfigure}
    \caption{Convergence performance with 20 agents (\ref{sub@fig:primal}) Primal variables. (\ref{sub@fig:objective}) Objective value. (\ref{sub@fig:violation}) Feasibility gap.}
    \label{fig:simulation}
\end{figure*}

\begin{theorem}
    Let Assumption~\ref{asm:convexity}-\ref{asm:weight_mat}.
    Then, under the conditions in Theorem~\ref{thm:convergence}, $\mathbf{x}^k = (x_1^k, \dots, x_N^k)$ converges to the optimal solution $\mathbf{x}^* = (x_1^*, \dots, x_N^*)$ of the problem~\eqref{eq:original_problem}.
\end{theorem}
\begin{IEEEproof}
    Since, from \eqref{eq:algo_primal} and strongly convexity of $\mathcal{L}_{\rho}^i$ with respect to $x_i$, $x_i^{k+1}$ uniquely minimizes $\mathcal{L}_\rho^i (x_i, p_i^k, q_i^k, v_i^k, z_i^k)$ with respect to $x_i$, we have that
    \begin{equation*}
        \begin{aligned}
            0 &\in \partial f_i(x_i^{k+1}) + A_i^T (p_i^k + \rho(A_i x_i^{k+1} - b_i - v_i^k)) \\
            &\quad + \sum_{l=1}^q [q_{il}^k + \rho (g_{il}(x_i^{k+1}) - z_{il}^k)]_+ \partial g_{il}(x_i^{k+1}) \\
            &= \partial f_i(x_i^{k+1}) + A_i^T u_i^{k+1} + \sum_{l=1}^q y_{il}^{k+1} \partial g_{il}(x_i^{k+1})
        \end{aligned}
    \end{equation*}
    where the update rules \eqref{eq:algo_dual_eq}-\eqref{eq:algo_dual_inq} are used.
    It is equivalent to
    \begin{equation}\label{eq:pf_x_i_optimality}
        x_i^{k+1} = \argmin \mathcal{L}_0^i(x, u_i^{k+1}, y_i^{k+1})
    \end{equation}
    and $\varphi_i(\zeta_i^{k+1}) = \mathcal{L}_0^i(x_i^{k+1}, u_i^{k+1}, y_i^{k+1})$.

    The result of Theorem~\ref{thm:convergence} implies that 
    \begin{equation*}
        \lim_{k \rightarrow \infty} \sum_i \varphi_i(\zeta_i^k) = \sum_i \varphi_i(\zeta^*),
    \end{equation*}
    so that there is a limit point $\boldsymbol{\zeta}' = (\zeta_1',\dots,\zeta_N')$ such that $\sum_i \varphi_i(\zeta_i') = \sum_i \varphi_i(\zeta^*)$.
    Let $x_i' = \arg\min_{x_i} \mathcal{L}_0^i(x_i, u_i', y_i')$ where $\zeta_i' = (u_i', y_i')$ for $i = 1,\dots,N$.
    Then, we have
    \begin{equation}\label{eq:pf_limit}
        \begin{aligned}
            \lim_{k \rightarrow \infty} \sum_i \mathcal{L}_0^i(x_i^k, u_i^k, y_i^k) 
            &= \sum_i \mathcal{L}_0^i(x_i', u_i', y_i') \\
            &= \sum_i \mathcal{L}_0^i(x_i^*, u_i^*, y_i^*).
        \end{aligned}
    \end{equation}
    From the optimality in \eqref{eq:pf_x_i_optimality}, we obtain
    \begin{equation*}
        \mathcal{L}_0^i(x_i', u_i', y_i') \le \mathcal{L}_0^i(x_i^*, u_i', y_i'),
    \end{equation*}
    and, from the saddle point theorem~\cite[Proposition 6.1.6]{bertsekas2016nonlinear}, we have
    \begin{equation*}
        \mathcal{L}_0^i(x_i^*, u_i', y_i') \le \mathcal{L}_0^i(x_i^*, u_i^*, y_i^*).
    \end{equation*}
    Then, from \eqref{eq:pf_limit}, we conclude $\sum_i \mathcal{L}_0^i(x_i', u_i', y_i') = \sum_i \mathcal{L}_0^i(x_i^*, u_i', y_i')$.
    Finally, a limit point $x_i'$ of $\{x_i^k\}_{k \in \mathbb{Z}_+}$ should be equal to $x_i^*$ for $i = 1,\dots,N$, due to strong convexity.
\end{IEEEproof}

\section{Numerical Example}\label{sec:simulation}
In this section, we demonstrate a numerical study of the proposed algorithm.
We consider a network of $N = 20$ agents where the objective is to solve the following optimization problem with randomly selected parameters $d_i \in \{1,\dots,5\}$ for $i \in \mathcal{V}$, $p = 25$, and $q = 1$: 
\begin{equation*}
    \begin{aligned}
        &\min_{\{x_i\}_{i \in \mathcal{V}}} &&\sum_{i \in \mathcal{V}} \frac{1}{2} x_i^T Q_i x_i + r_i^T x_i + \norm{x_i}_1 \\
        &\text{subject to} && \sum_i A_i x_i = \sum_i b_i \\
        &&& \sum_i \norm{x_i - a_i}^2 \le \sum_i c_i^2.
    \end{aligned}
\end{equation*}
The coefficient $Q_i \in \mathbb{R}^{d_i \times d_i}$ is random symmetric positive definite, and $r_i \in \mathbb{R}^{d_i}$ is a random vector.
The constraints are randomly generated under conditions where the feasibility set is not an empty set.
The initial value of $\mathbf{x}^0$ is randomly selected, and the initial value of auxiliary variables and multipliers are initialized to zero.

In \figurename~\ref{fig:simulation}, we plot the evolution of (\ref{sub@fig:primal}) the error of the primal variables, (\ref{sub@fig:objective}) the objective function value, and (\ref{sub@fig:violation}) the constraint violation over iterations.
The error of the primal variable and objective function value are defined as $\norm{x_i - x_i^*}, \forall i \in \mathcal{V}$ and $|f(\mathbf{x}^k) - f^*|$, respectively.
The constraint violations are defined as $\max_{i \in \mathcal{V}}\norm{A_i x_i - b_i - v_i}$ and $\max \{ \norm{x_i - a_i}^2 - c_i^2 - z_i, 0 \}$.
\figurename~\ref{fig:simulation} indicates that the proposed algorithm has remarkable convergence property with respect to both of optimality and feasibility on time-varying directed network.

\section{Conclusion}\label{sec:conclusion}
We developed a distributed algorithm designed to solve nonsmooth convex optimization problems that contain both network-wise equality and inequality constraints.
The algorithm is constructed by introducing auxiliary variables for decomposition and utilizing the primal-dual method.
Under the conditions of the strong convexity in the local objective functions and bounded subdifferential of the inequality constraints, we prove that the proposed algorithm achieves an $O(1/k)$ rate of convergence in dual aspect on time-varying directed graphs.
Furthermore, simulations were conducted to verify the algorithm's advantages.
In future work, we plan to extend the proposed algorithm to asynchronous time-varying networks and investigate the way to weaken the doubly stochastic assumption.

\appendices
\section{Preliminaries for Proofs}
\subsection{Monotone Operator and Subdifferential}
We use $\mathcal{H}$ to denote the Euclidean space, and $\Gamma_0(\mathcal{H})$ the class of proper lower semicontinuous convex function from $\mathcal{H}$ to $(-\infty, +\infty]$.
An operator $M$ on a Euclidean space $\mathcal{H}$ is a set-valued mapping, i.e., $M: \mathcal{H} \rightarrow 2^\mathcal{H}$.
The graph of $M$ is defined as $\gra M = \{(x, y) \in \mathcal{H} \times \mathcal{H} \mid y \in Mx\}$.
The operator $M$ is called monotone if 
\begin{equation*}
    \langle x - x', y - y' \rangle \ge 0,
\end{equation*}
and $\mu$-strongly monotone if
\begin{equation}\label{eq:prel_strong_monotone}
    \langle x - x', y - y' \rangle \ge \mu \lVert x - x' \rVert^2,
\end{equation}
for $(x, y), (x', y') \in \gra M$.
A monotone operator is said to be maximal if there is no monotone operator $T'$ such that $\gra T \subset \gra T'$.
The operator $M$ is called cyclically monotone \cite{bauschke2011convex} if for every integer $n \ge 2$,
\begin{equation*}
    \sum_{i = 1}^n \langle x_{i+1} - x_i, u_i \rangle \le 0,
\end{equation*}
where $(x_i, u_i) \in \gra M, i = 1, \dots, n$ and $x_{n+1} = x_1$.
The subdifferential $\partial f: \mathcal{H} \rightarrow 2^\mathcal{H}$ of $f \in \Gamma_0(\mathcal{H})$ is the set-valued operator defined as
\begin{equation*}
    \partial f(x) = \{p \in \mathcal{H} \mid f(y) \ge f(x) + \langle p, y - x \rangle, \forall y \in \mathcal{H}\}.
\end{equation*}
The subdifferential $\partial f$ of $f \in \Gamma_0(\mathcal{H})$ is maximal monotone and cyclical monotone \cite[Propostion 20.25, 22.14]{bauschke2011convex}.

\subsection{Convex Analysis}
A differentiable function $f: \mathcal{H} \rightarrow \mathbb{R}$ is said to be $L$-smooth if it has Lipschitz continuous gradient with coefficient $L > 0$, i.e.,
\begin{equation}\label{eq:prel_Lipschitz_gradient}
    \lVert \nabla f(x) - \nabla f(x') \rVert \le L \lVert x - x' \rVert,
\end{equation}
for $\forall x, x' \in \mathcal{H}$.
For an $L$-smooth function $f$, the followings hold \cite{zhou2018fenchel}:
\begin{align}
    \langle \nabla f(x) - \nabla f(x'), x - x' \rangle \le L \lVert x - x' \rVert^2, \label{eq:prel_smooth_ineq_1}\\
    f(x') \le f(x) + \langle \nabla f(x), x' - x \rangle + \frac{L}{2} \lVert x' - x \rVert^2, \label{eq:prel_smooth_ineq_2}
\end{align}
for $\forall x, x' \in \mathcal{H}$.
A function $g: \mathcal{H} \rightarrow \mathbb{R}$ is said to be convex if for $x, y \in \mathcal{H}$
\begin{equation}\label{eq:prel_convex}
    g(y) \ge g(x) + \langle s, y - x \rangle
\end{equation}
and $\mu$-strongly convex if
\begin{equation*}
    g(y) \ge g(x) + \langle s, y - x \rangle + \frac{\mu}{2}\lVert y - x \rVert^2
\end{equation*}
where $s \in \partial g(x)$.
Since the subdifferential $\partial g$ of a $\mu$-strongly convex function $g$ is $\mu$-strongly monotone \cite{bauschke2011convex}, from \eqref{eq:prel_strong_monotone}, we have
\begin{equation*}
    \mu \lVert x - x' \rVert \le \lVert s - s' \rVert,
\end{equation*}
where $s \in \partial g(x)$ and $s' \in \partial g(x')$.
Let $x^*$ be a minimizer of a convex function $f: \mathcal{H} \rightarrow \mathbb{R}$ over a convex set $X \subset \mathcal{H}$.
Then, there exists a subgradient $d \in \partial f(x^*)$ satisfying
\begin{equation} \label{eq:prel_optimal_cond}
    \langle x - x^*, d \rangle \ge 0,\quad \forall x \in X.
\end{equation}

\subsection{Convergence Analysis Methodology}
We summarize a methodology for convergence analysis called aggregate lower-bounding (ALB) proposed in \cite{shahriari2025double}.
In the convergence analysis based on Lyapunov method, one seeks a Lyapunov function satisfying
\begin{equation*}
    L(x^{k+1}) - L(x^k) \le - V(x^k)
\end{equation*}
where $V$ is a positive definite function.

As the assumptions become harsher and the algorithm becomes more complex, it becomes more difficult to find an appropriate Lyapunov function that satisfies the above conditions.
In ALB, the Lyapunov function is replaced by a negative function $D_k$ as
\begin{equation*}
    V(x^k) + D_k \le 0.
\end{equation*}
This approach sums over all iteration:
\begin{equation}\label{eq:prel_aggregate}
    \sum_{k = 0}^{K-1} V(x^k) + \bar{D}_K \le 0
\end{equation}
where $\bar{D}_K = \sum_{k=0}^K D_k$, and show that there exists a lower bound $C$ for the aggregate term $\bar{D}_K$ regardless of $K$.
Then, we obtain
\begin{equation*}
    \sum_{k = 0}^{K-1} V(x^k) \le - C,
\end{equation*}
and if $V$ is convex, we can conclude $V(\bar{x}^K) \le - C/K$ where $\bar{x}^K = \frac{1}{K}\sum_{k = 0}^{K-1} x^k$, yielding the standard $O(1/K)$ rate.

\section{Proof of Theorem~\ref{thm:convergence}}\label{sec:convergence_proof}
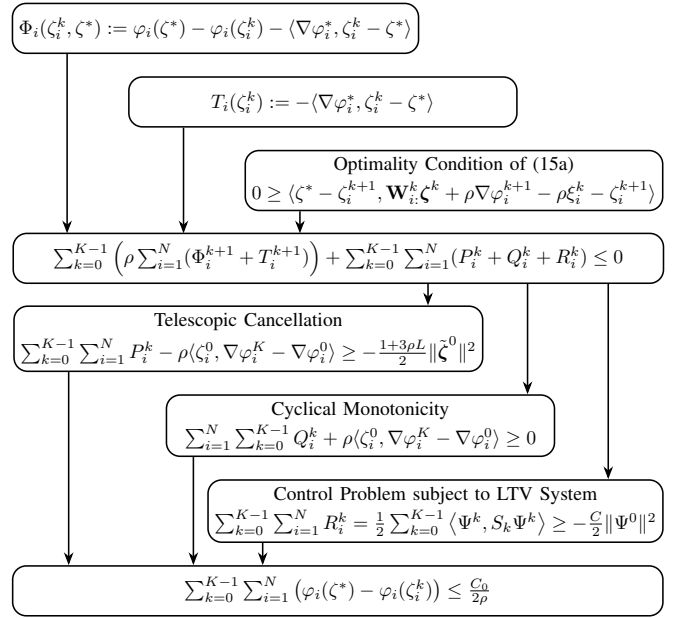
\begin{figure}
    \centering
    \resizebox{\columnwidth}{!}{
    \begin{tikzpicture}[node distance=.4cm and 2.1cm,>=Stealth]
    \node[block] (Phi) {$\Phi_i(\zeta_i^k, \zeta^*) := \varphi_i(\zeta^*) - \varphi_i(\zeta_i^k) - \langle \nabla\varphi_i^*, \zeta_i^k - \zeta^* \rangle$};
    \node[block, below right=of Phi.south west] (T) {$T_i(\zeta_i^k) := -\langle \nabla \varphi_i^*, \zeta_i^k - \zeta^* \rangle$};
    \node[block, below right=of T.south west] (Proj) {\shortstack{Optimality Condition of \eqref{eq:dual_update} \\ $0 \ge \langle \zeta^* - \zeta_i^{k+1}, \mathbf{W}_{i:}^k \boldsymbol{\zeta}^k + \rho \nabla\varphi_i^{k+1} - \rho \xi_i^k - \zeta_i^{k+1} \rangle$}};

    \path let
        \p1=(Phi.west),
        \p2=(Proj.east)
    in node[block, below=of Proj.south east, anchor=north east, minimum width=\x2-\x1-\pgflinewidth, outer sep=0pt] (S) {$\sum_{k = 0}^{K-1} \left(\rho \sum_{i = 1}^N (\Phi_i^{k+1} + T_i^{k+1})\right) + \sum_{k = 0}^{K-1} \sum_{i = 1}^N (P_i^k + Q_i^k + R_i^k) \le 0$};

    \node[block, below=of S.south west, anchor=north west] (P) {\shortstack{Telescopic Cancellation \\ $\sum_{k=0}^{K-1} \sum_{i = 1}^N P_i^k - \rho \langle \zeta_i^0, \nabla\varphi_i^K - \nabla\varphi_i^0 \rangle \ge - \frac{1+3\rho L}{2}\lVert \tilde{\boldsymbol{\zeta}}^0 \rVert^2$}};
    \node[block, below right=of P.south west, xshift=.65cm] (Q) {\shortstack{Cyclical Monotonicity \\ $\sum_{i = 1}^N \sum_{k = 0}^{K-1} Q_i^k + \rho \langle \zeta_i^0, \nabla\varphi_i^K - \nabla\varphi_i^0 \rangle \ge 0$}};
    \node[block, below right=of Q.south west, xshift=-1.35cm] (R) {\shortstack{Control Problem subject to LTV System \\ $\sum_{k = 0}^{K-1}\sum_{i = 1}^N R_i^k = \frac{1}{2}\sum_{k = 0}^{K-1}\left\langle \Psi^k, S_k \Psi^k \right\rangle \ge - \frac{C}{2} \lVert \Psi^0 \rVert^2$}};
    \path let
        \p1=(Phi.west),
        \p2=(Proj.east)
    in node[block, below=of R.south east, anchor=north east, minimum width=\x2-\x1-\pgflinewidth] (All) {$\sum_{k = 0}^{K-1} \sum_{i=1}^N \left( \varphi_i(\zeta^*) - \varphi_i(\zeta_i^k) \right) \le \frac{C_0}{2\rho}$};

    \path[line] ([xshift=1cm]Phi.south west) -- ([xshift=1cm]Phi.south west |- S.north);
    \path[line] ([xshift=1cm]T.south west) -- ([xshift=1cm]T.south west |- S.north);
    \path[line] ([xshift=1cm]Proj.south west) -- ([xshift=1cm]Proj.south west |- S.north);

    \path[line] ([xshift=-1cm]S.south -| P.north east) -- ([xshift=-1cm]P.north east);
    \path[line] ([xshift=-.5cm]S.south -| Q.north east) -- ([xshift=-.5cm]Q.north east);
    \path[line] ([xshift=-1cm]S.south -| R.north east) -- ([xshift=-1cm]R.north east);

    \path[line] ([xshift=1cm]P.south west) -- ([xshift=1cm]P.south west |- All.north);
    \path[line] ([xshift=.5cm]Q.south west) -- ([xshift=.5cm]Q.south west |- All.north);
    \path[line] ([xshift=1cm]R.south west) -- ([xshift=1cm]R.south west |- All.north);
\end{tikzpicture}
}

    \caption{Proof Diagram of Theorem~\ref{thm:convergence}}
    \label{fig:pf_diagram}
\end{figure}
The following proof is based on the procedure in convergence analysis in \cite[Section III-C]{shahriari2025double}.
In contrast to \cite{shahriari2025double}, our algorithm is on the time-varying networks, so the method for finding the lower bound of the aggregate term is different.
Figure~\ref{fig:pf_diagram} presents the flow of this section.
Based on the definitions of the two functions $\Phi_i^k$, $T_i^k$ and the optimal condition of \eqref{eq:dual_update}, we derive an equation of the form~\eqref{eq:prel_aggregate} then find the lower bounds for three aggregate terms $P_i^k$, $Q_i^k$, and $R_i^k$ through the telescopic cancellation, cyclic monotonicity, and optimization of control problems, respectively.

\subsection{Deriving the aggregate term}
We firstly define $\Phi_i(\zeta_1, \zeta_2) := - \varphi_i(\zeta_1) + \varphi_i(\zeta_2) + \left\langle \nabla \varphi_i(\zeta_2), \zeta_1 - \zeta_2 \right\rangle$, $\Phi_i^k := \Phi_i(\zeta_i^k, \zeta^*)$, $\nabla \varphi_i^* := \nabla \varphi_i(\zeta^*)$, and $\nabla \varphi_i^k := \nabla \varphi_i(\zeta^k)$.
Note that since the dual $\varphi_i$ is concave, from the convexity~\eqref{eq:prel_convex} of $-\varphi_i$, we have $\Phi_i^k \ge 0$ and
\begin{equation}\label{eq:pf_bregman}
    \begin{aligned}
        - \Phi_i(\zeta^*, \zeta_i^k) 
        &= \varphi_i(\zeta^*) - \varphi_i(\zeta_i^k) - \left\langle \nabla\varphi_i^k, \zeta^* - \zeta_i^k \right\rangle \\
        &= \Phi_i^k + \left\langle \nabla\varphi_i^k - \nabla\varphi_i^*, \zeta_i^k - \zeta^* \right\rangle
        \le 0.
    \end{aligned}
\end{equation}
From Lemma~\ref{lem:dual_Lipschitz} and \eqref{eq:prel_smooth_ineq_2}, $\varphi_i$ is $L$-smooth and 
\begin{equation}\label{eq:pf_Lipschitz}
    \begin{aligned}
        \frac{L}{2} \lVert \zeta_i^{k+1} - \zeta_i^k \rVert^2 
        &\ge \varphi_i(\zeta_i^k) - \varphi_i(\zeta_i^{k+1}) + \left\langle \nabla\varphi_i^k, \zeta_i^{k+1} - \zeta_i^k \right\rangle \\
        &= \Phi_i^{k+1} - \Phi_i^k + \left\langle \nabla\varphi_i^k - \nabla\varphi_i^*, \zeta_i^{k+1} - \zeta_i^k \right\rangle.
    \end{aligned}
\end{equation}
Adding~\eqref{eq:pf_bregman} and \eqref{eq:pf_Lipschitz} yields
\begin{equation} \label{eq:pf_from_bregman}
    \Phi_i^{k+1} + \left\langle \nabla\varphi_i^k - \nabla\varphi_i^*, \zeta_i^{k+1} - \zeta^* \right\rangle - \frac{L}{2} \lVert \zeta_i^{k+1} - \zeta_i^k \rVert^2 \le 0.
\end{equation}

Second, we define 
\begin{equation*}
    T_i(\zeta) := - \left\langle \nabla\varphi_i^*, \zeta - \zeta^* \right\rangle,
\end{equation*}
and $T_i^k := T_i(\zeta_i^k)$. 
Note that $\nabla\varphi_i^* = (A_i x_i^* - b_i - v_i^*, g_i(x_i^*) - z_i^*) = (0, g_i(x_i^*) - z_i^*)$, from the KKT condition~\eqref{eq:kkt_local}, $\langle y_i^*, g_i(x_i^*) - z_i^* \rangle = 0$ and $g_i(x_i^*) \le z_i^*$.
Since the dual variable $\zeta_i^k$ is always in $\mathbb{R}^p \times \mathbb{R}_+^q$ for $i \in \mathcal{V}$ and $k \in \mathbb{Z}_+$, i.e., $y_i^k \ge 0$, we have $T_i^k = - \langle g_i(x_i^*) - z_i^*, y_i^k \rangle \ge 0$.

Based on the optimality condition~\eqref{eq:prel_optimal_cond}, the equation~\eqref{eq:dual_update} implies
\begin{equation}\label{eq:pf_proj}
    \begin{aligned}
        0 &\ge
        \langle \zeta^* - \zeta_i^{k+1}, \mathbf{W}_{i:}^k \boldsymbol{\zeta}^k + \rho \nabla\varphi_i^{k+1} - \rho \xi_i^k - \zeta_i^{k+1} \rangle \\
        &= \frac{1}{2} \lVert \zeta_i^{k+1} - \zeta^* \rVert^2 - \frac{1}{2} \lVert \mathbf{W}_{i:}^k \boldsymbol{\zeta}^k - \zeta^* \rVert^2 \\
        &\quad + \frac{1}{2} \lVert \zeta_i^{k+1} - \mathbf{W}_{i:}^k \boldsymbol{\zeta}^k \rVert^2 + \rho \langle \zeta^* - \zeta_i^{k+1}, \nabla\varphi_i^{k+1} - \xi_i^k \rangle
    \end{aligned}
\end{equation}
using $2\langle a, a-b \rangle = \lVert a \rVert^2 - \lVert b \rVert^2 + \lVert a - b \rVert^2$.
Adding $T_i^{k+1} + \langle \nabla \varphi^*, \zeta_i^{k+1} - \zeta^* \rangle = 0$ and \eqref{eq:pf_proj}, we have
\begin{equation} \label{eq:pf_normal}
    \begin{aligned}
        \rho T_i^{k+1} + \rho \left\langle \zeta^* - \zeta_i^{k+1}, \nabla\varphi_i^{k+1} - \nabla\varphi_i^* - \xi_i^k \right\rangle \\
        + \frac{1}{2} \lVert \zeta_i^{k+1} - \zeta^* \rVert^2 - \frac{1}{2} \lVert \mathbf{W}_{i:}^k \boldsymbol{\zeta}^k - \zeta^* \rVert^2 \\
        + \frac{1}{2} \lVert \zeta_i^{k+1} - \mathbf{W}_{i:}^k \boldsymbol{\zeta}^k \rVert^2
        \le 0.
    \end{aligned}
\end{equation}
Multiplying \eqref{eq:pf_from_bregman} by $\rho$ and adding \eqref{eq:pf_normal}, 
we have
\begin{equation*}
    \rho(\Phi_i^{k+1} + T_i^{k+1}) + P_i^k + Q_i^k + R_i^k \le 0
\end{equation*}
where 
\begin{equation*}
    \begin{aligned}
        P_i^k &= \frac{1}{2}\lVert\zeta_i^{k+1} - \zeta^* \rVert^2 - \frac{1}{2}\lVert \mathbf{W}_{i:}^k \boldsymbol{\zeta}^k  - \zeta^* \rVert^2 \\
        &\qquad + \rho \left\langle \zeta^*, \nabla\varphi_i^{k+1} - \nabla\varphi_i^k \right\rangle, \\
        Q_i^k &= - \rho\left\langle \zeta_i^{k+1}, \nabla\varphi_i^{k+1} - \nabla\varphi_i^k \right\rangle, \\
        R_i^k &= \frac{1}{2}\lVert \zeta_i^{k+1} - \mathbf{W}_{i:}^k \boldsymbol{\zeta}^k \rVert^2 - \frac{\rho L}{2}\lVert \zeta_i^{k+1} - \zeta_i^k \rVert^2 \\
        &\qquad - \rho \left\langle \zeta^* - \zeta_i^{k+1}, \xi_i^k \right\rangle.
    \end{aligned}
\end{equation*}
Summing over $i = 1, \dots, N$ and $k = 0, \dots, K-1$, we have
\begin{equation*}
    \sum_{k = 0}^{K-1} \left(\rho \sum_{i = 1}^N (\Phi_i^{k+1} + T_i^{k+1})\right) + \sum_{k = 0}^{K-1} \sum_{i = 1}^N (P_i^k + Q_i^k + R_i^k) \le 0.
\end{equation*}
The first summation $\sum_{k = 0}^{K-1} \left(\rho \sum_{i = 1}^N (\Phi_i^{k+1} + T_i^{k+1})\right)$ is equivalent to $\sum V_k$, and the term consisting of $P_i^k$, $Q_i^k$ and $R_i^k$ is equivalent to $\bar{D}_K$ in \eqref{eq:prel_aggregate}.
We now show that each terms in the last summation have a lower bound.

\subsection{Bounding the aggregate term}
\subsubsection{\protect First term $P_i^k$}
From the nonexpansive property of doubly stochastic matrices, we have
\begin{equation*}
    \lVert \mathbf{W}^k \boldsymbol{\zeta}^k - 1\otimes\zeta^* \rVert \le \lVert \boldsymbol{\zeta}^k - 1\otimes\zeta^* \rVert.
\end{equation*}
Then, we obtain 
\begin{equation*}
    \begin{aligned}
        \sum_{k = 0}^{K-1} \sum_{i = 1}^N P_i^k 
        &\ge \frac{1}{2}\lVert \boldsymbol{\zeta}^K - 1\otimes\zeta^* \rVert^2 - \frac{1}{2}\lVert \boldsymbol{\zeta}^0 - 1\otimes\zeta^* \rVert^2 \\
        &\quad + \rho \sum_{i = 1}^N \left\langle \zeta^*, \nabla\varphi_i^K - \nabla\varphi_i^0 \right\rangle.
    \end{aligned}
\end{equation*}
Subtracting $\rho \sum_i \langle \zeta_i^0, \nabla\varphi_i^K - \nabla\varphi_i^0 \rangle$, we have
\begin{equation*}
    \begin{aligned}
        &\sum_{i = 1}^N \left( \sum_{k=0}^{K-1}  P_i^k - \rho \langle \zeta_i^0, \nabla\varphi_i^K - \nabla\varphi_i^0 \rangle \right) \\
        &\ge \frac{1}{2}\lVert \boldsymbol{\zeta}^K - 1\otimes\zeta^* \rVert^2 - \frac{1}{2}\lVert \boldsymbol{\zeta}^0 - 1\otimes\zeta^* \rVert^2 \\
        &\; + \rho \sum_{i = 1}^N \left( \langle \zeta^* - \zeta_i^0, \nabla\varphi_i^* - \nabla\varphi_i^0 \rangle + \langle \zeta^* - \zeta_i^0, \nabla\varphi_i^K - \nabla\varphi_i^* \rangle \right) 
    \end{aligned}
\end{equation*}
Since, from \eqref{eq:prel_smooth_ineq_1},
\begin{equation*}
    \langle \zeta^* - \zeta_i^0, \nabla\varphi_i^* - \nabla\varphi_i^0 \rangle \ge - L \lVert \zeta^* - \zeta_i^0 \rVert^2,
\end{equation*}
and, based on Cauchy–Schwarz inequality and \eqref{eq:prel_Lipschitz_gradient},
\begin{equation*}
    \begin{aligned}
        \langle \zeta^* - \zeta_i^0, \nabla\varphi_i^K - \nabla\varphi_i^* \rangle 
        &\ge - \lVert \zeta^* - \zeta_i^0 \rVert \lVert \nabla\varphi_i^K - \nabla\varphi_i^* \rVert \\
        &\ge - L \lVert \zeta^* - \zeta_i^0 \rVert \lVert \zeta_i^K - \zeta^* \rVert,
    \end{aligned}
\end{equation*}
we have 
\begin{equation*}
    \begin{aligned}
        &\sum_{i = 1}^N \left( \sum_{k=0}^{K-1}  P_i^k - \rho \langle \zeta_i^0, \nabla\varphi_i^K - \nabla\varphi_i^0 \rangle \right) \\
        &\ge \frac{1}{2}\lVert \boldsymbol{\zeta}^K - 1\otimes\zeta^* \rVert^2 - \frac{1 + 2\rho L}{2}\lVert \boldsymbol{\zeta}^0 - 1\otimes\zeta^* \rVert^2 \\
        &\quad - \rho L \lVert \boldsymbol{\zeta}^K - 1\otimes\zeta^* \rVert \lVert \boldsymbol{\zeta}^0 - 1\otimes\zeta^* \rVert \\
        &= \frac{1 - \rho L}{2}\lVert \boldsymbol{\zeta}^K - 1\otimes\zeta^* \rVert^2 - \frac{1 + 3\rho L}{2}\lVert \boldsymbol{\zeta}^0 - 1\otimes\zeta^* \rVert^2  \\
        &\quad + \frac{\rho L}{2}\left(\lVert \boldsymbol{\zeta}^K - 1\otimes\zeta^* \rVert - \lVert \boldsymbol{\zeta}^0 - 1\otimes\zeta^* \rVert\right)^2 \\
        &\ge \frac{1 - \rho L}{2}\lVert \boldsymbol{\zeta}^K - 1\otimes\zeta^* \rVert^2 - \frac{1 + 3\rho L}{2}\lVert \boldsymbol{\zeta}^0 - 1\otimes\zeta^* \rVert^2.
    \end{aligned}
\end{equation*}
Then, the summation of the first term $P_i^k$ has a lower bound as
\begin{equation*}
    \sum_{i = 1}^N \left( \sum_{k=0}^{K-1}  P_i^k - \rho \langle \zeta_i^0, \nabla\varphi_i^K - \nabla\varphi_i^0 \rangle \right)
    \ge - \frac{1 + 3\rho L}{2}\lVert \boldsymbol{\zeta}^0 - 1\otimes\zeta^* \rVert^2.
\end{equation*}

\subsubsection{Second term $Q_i^k$}
From the cyclical monotonicity of $- \nabla \varphi_i$, we have
\begin{equation*}
    \begin{aligned}
        &\sum_{k = 0}^{K-1} Q_i^k + \rho \langle \zeta_i^0, \nabla\varphi_i^K - \nabla\varphi_i^0 \rangle \\
        &= \rho \sum_{k = 0}^{K-1} \langle \zeta_i^{k+1}, \nabla\varphi_i^k - \nabla\varphi_i^{k+1}\rangle + \rho \langle \zeta_i^0, \nabla\varphi_i^K - \nabla\varphi_i^0 \rangle \\
        &= \rho \langle \zeta_i^0 - \zeta_i^K, \nabla\varphi_i^K \rangle + \rho \sum_{k = 0}^{K-1} \langle \zeta_i^{k+1} - \zeta_i^k, \nabla\varphi_i^k \rangle \ge 0.
    \end{aligned}
\end{equation*}

\subsubsection{Third term $R_i^k$}
The third term can be written as
\begin{equation*}
    \sum_{k = 0}^{K-1}\sum_{i = 1}^N R_i^k = \frac{1}{2}\sum_{k = 0}^{K-1}\left\langle \Psi^k, S_k \Psi^k \right\rangle
\end{equation*}
where $\Psi^k = (\tilde{\boldsymbol{\zeta}}^{k+1}, \tilde{\boldsymbol{\zeta}}^k, \rho\boldsymbol{\xi}^k)$, $\tilde{\boldsymbol{\zeta}}^k = \boldsymbol{\zeta}^k - 1_N \otimes \zeta^*$ and 
\begin{equation*}
    S_k = 
    \begin{bmatrix}
        (1 - \rho L) I_N & 2 (\rho L I - W_k) & 2 I_N \\
        0 & {W^k}^T W^k - \rho L I_N & 0 \\
        0 & 0 & 0
    \end{bmatrix} \otimes I_{p+q}.
\end{equation*}

We show that the third term has a lower bound using a control problem with a linear time-varying difference system.
Consider the following control problem with $C > 0$
\begin{equation} \label{eq:pf_opt}
    \begin{aligned}
        \min_{\substack{\{\Psi_k\}_{k = 0}^{K-1} \\ \left\{\tilde{\boldsymbol{\zeta}}^{k+2}\right\}_{k = 0}^{K-2}}}&
        \frac{1}{2}\sum_{k = 0}^{K-1}\left\langle \Psi^k, S_k \Psi^k \right\rangle + \frac{C}{2}\lVert \Psi^0 \rVert^2 \\
    \end{aligned}
\end{equation}
subject to, for $k = 0,\dots,K-2$, the linear time-varying difference system
\begin{equation}\label{eq:pf_linear}
    \Psi^{k+1} = U_{k+1} \Psi^k + P \tilde{\boldsymbol{\zeta}}^{k+2}
\end{equation}
where 
\begin{equation*}
    U_{k+1} = 
    \begin{bmatrix}
        0 & 0 & 0 \\
        I & 0 & 0 \\
        I - W^{k+1} & 0 & I
    \end{bmatrix}
    \otimes I_{p+q},\quad
    P = 
    \begin{bmatrix}
        I \\ 0 \\ 0
    \end{bmatrix}.
\end{equation*}
We want to show that the optimal solution of the above optimization is only zero, so that the optimal value is zero.
If our claim is false, the above control problem is unbounded and then the following restricted optimization has a strictly negative optimal value at a non-zero solution:
\begin{equation*}
    \begin{aligned}
        \min_{\substack{\{\Psi_k\}_{k = 0}^{K-1} \\ \left\{\tilde{\boldsymbol{\zeta}}^{k+2}\right\}_{k = 0}^{K-2}}}&
        \frac{1}{2}\sum_{k = 0}^{K-1}\left\langle \Psi^k, S_k \Psi^k \right\rangle + \frac{C}{2}\lVert \Psi^0 \rVert^2 \\
        \textrm{subject to}\;\;& \eqref{eq:pf_linear}, \qquad k = 0,\dots,K-2, \\
        &\frac{1}{2}\lVert \Psi^0 \rVert^2 + \frac{1}{2}\sum_{k = 0}^{K-2} \left\lVert \tilde{\boldsymbol{\zeta}}^{k+2} \right\rVert^2 \le \frac{1}{2}.
    \end{aligned}
\end{equation*}
Then, there is non-zero solution satisfying the KKT condition with corresponding multipliers $\{\Lambda^k\}, \beta \ge 0$ such that
\begin{equation}\label{eq:pf_kkt_system}
    \begin{aligned}
        S_{k+1} \Psi^{k+1} - U_{k+1}^T \Lambda^{k+1} + \Lambda^k &= 0, \\
        \Psi^{k+1} - U_{k+1} \Psi^k - P \tilde{\boldsymbol{\zeta}}^{k+2} &= 0, \\
        - P^T \Lambda^k + \beta \tilde{\boldsymbol{\zeta}}^{k+2} &= 0,
    \end{aligned}
\end{equation}
for $k = 0,1,\dots,K-2$ with the boundary conditions $\Lambda_{K-1} = 0$ and $(S_0 + (C + \beta) I) \Psi^0 - U_1^T \Lambda^0 = 0$ which is equivalent to
\begin{subequations}
    \begin{align}
        (1-\rho L + C + \beta) \tilde{\boldsymbol{\zeta}}^1 + 2(\rho L I - W^0) \tilde{\boldsymbol{\zeta}}^0 + 2 \rho \boldsymbol{\xi}^0 \notag \\
        - \lambda_2^0 - (I - {W^1}^T) \lambda_3^0 &= 0, \label{eq:pf_zeta_1} \\
        ({W^0}^T W^0 + (C + \beta - \rho L)I) \tilde{\boldsymbol{\zeta}}^0 &= 0, \label{eq:pf_zeta_0} \\
        (C + \beta)\rho \boldsymbol{\xi}^0 - \lambda_3^0 &= 0. \label{eq:pf_xi_0}
    \end{align}
\end{subequations}
Plugging $S_k, U_k, P$ with $\Lambda^k = (\lambda_1^k, \lambda_2^k, \lambda_3^k)$, the equation~\eqref{eq:pf_kkt_system} becomes
\begin{subequations}\label{eq:pf_kkt_long}
    \begin{align}
        (1-\rho L) \tilde{\boldsymbol{\zeta}}^{k+2} + 2(\rho L I - W^k) \tilde{\boldsymbol{\zeta}}^{k+1} + 2 \rho \boldsymbol{\xi}^{k+1} \\
        - \lambda_2^{k+1} - (I - {W^{k+1}}^T) \lambda_3^{k+1} + \lambda_1^k &= 0, \\
        ({W^{k+1}}^T W^{k+1} - \rho L I) \tilde{\boldsymbol{\zeta}}^{k+1} + \lambda_2^k &=0, \label{eq:pf_kkt_zeta_lambda2} \\
        -\lambda_3^{k+1} + \lambda_3^k &= 0, \label{eq:pf_lambda_3} \\
        - \lambda_1^k + \beta \tilde{\boldsymbol{\zeta}}^{k+2} &= 0, \\
        \rho \boldsymbol{\xi}^{k+1} - \rho \boldsymbol{\xi}^k - (I - W^{k+1}) \tilde{\boldsymbol{\zeta}}^{k+1} &= 0,
    \end{align}
\end{subequations} 
for $k = 0,1,\dots,K-2$.
Since $\Lambda_{K-1} = 0$ and \eqref{eq:pf_lambda_3}, we have $\lambda_3^k \equiv 0$, and then, with $W^0 = I$, it follows $\rho \boldsymbol{\xi}^0 = \tilde{\boldsymbol{\zeta}}^0 = 0$ according to \eqref{eq:pf_zeta_0} and \eqref{eq:pf_xi_0}.
Removing $\lambda_2^0$ in \eqref{eq:pf_zeta_1} and \eqref{eq:pf_kkt_zeta_lambda2} at $k = 0$, we obtain
\begin{equation*}
    ({W^{1}}^T W^{1} + (1 - 2\rho L + C + \beta) I) \tilde{\boldsymbol{\zeta}}^{1} = 0,
\end{equation*}
and then $\tilde{\boldsymbol{\zeta}}^{1} = 0$.
Therefore, removing $\Lambda^k$, the equations~\eqref{eq:pf_kkt_long} become
\begin{equation*}
    \begin{aligned}
        ({W^{k+2}}^T W^{k+2} + (1-2\rho L+\beta)I) \tilde{\boldsymbol{\zeta}}^{k+2} \\
        + 2(\rho L I - W^k) \tilde{\boldsymbol{\zeta}}^{k+1} + 2 \rho \boldsymbol{\xi}^{k+1} &= 0, \\
        \rho \boldsymbol{\xi}^{k+2} - \rho \boldsymbol{\xi}^{k+1} - (I - W^{k+2}) \tilde{\boldsymbol{\zeta}}^{k+2} &= 0,
    \end{aligned}
\end{equation*}
for $k = 0,\dots,K-3$, with the boundary conditions $\tilde{\boldsymbol{\zeta}}^1 = \rho \boldsymbol{\xi}^1 = 0$ and 
\begin{equation*}
    \begin{aligned}
        (1-\rho L + \beta) \tilde{\boldsymbol{\zeta}}^K + 2 \rho \boldsymbol{\xi}^{K-1} \\
        + 2(\rho L I - W^{K-2}) \tilde{\boldsymbol{\zeta}}^{K-1} &= 0.
    \end{aligned}
\end{equation*}
Since $1 - 2 \rho L + \beta > 0$ for every $\beta > 0$ from the assumption $\rho \le 1/2L$, we obtain $\tilde{\boldsymbol{\zeta}}^{k+2} = 0$ for $k = 0,\dots,K-2$ so that the system~\eqref{eq:pf_kkt_system} has no non-zero solution.
Therefore, the optimization~\eqref{eq:pf_opt} has zero optimal value, and then
\begin{equation*}
    \sum_{k = 0}^{K-1}\sum_{i = 1}^N R_i^k = \frac{1}{2}\sum_{k = 0}^{K-1}\left\langle \Psi^k, S_k \Psi^k \right\rangle \ge - \frac{C}{2} \lVert \Psi^0 \rVert^2.
\end{equation*}

Finally, We obtain
\begin{equation*}
    \sum_{k = 0}^{K-1} \left(\rho \sum_{i = 1}^N (\Phi_i^{k+1} + T_i^{k+1})\right) 
    \le \frac{1 + 3\rho L}{2}\lVert \tilde{\boldsymbol{\zeta}}^0 \rVert^2 + \frac{C}{2} \lVert \Psi^0 \rVert^2.
\end{equation*}
Since $\Phi_i^k + T_i^k = \varphi_i(\zeta^*) - \varphi_i(\zeta^k)$, we conclude
\begin{equation*}
    \sum_{k = 0}^{K-1} \sum_{i=1}^N \left( \varphi_i(\zeta^*) - \varphi_i(\zeta_i^k) \right) \le \frac{C_0}{2\rho},
\end{equation*}
where $C_0 = (1 + 3\rho L)\lVert \tilde{\boldsymbol{\zeta}}^0 \rVert^2 + C \lVert \Psi^0 \rVert^2 + 2\rho \sum_i \varphi_i(\zeta_i^0)$.

\bibliographystyle{ieeetr}
\bibliography{reference}

\end{document}